\documentclass[12pt]{amsart}
\usepackage{ifpdf}
\usepackage{amsfonts, amsbsy, amsmath, amssymb, latexsym}
\usepackage{etoolbox}
\usepackage{mathrsfs,array}
\usepackage[english]{babel}
\usepackage{amsthm}
\usepackage{amssymb}
\usepackage{amsfonts, dsfont, graphicx, srcltx}
\usepackage{tikz}
\usepackage{graphicx}
\usepackage{hyperref}
\usepackage{verbatim}
\usepackage{fancyvrb}
\usepackage{xcolor}
\usepackage{verbatim}
\usepackage{float}
\usepackage{pgfplots}
\pgfplotsset{compat=1.17}
\usepackage{xspace}
\newtheorem{theorem}{Theorem}[section]
\newtheorem{lemma}[theorem]{Lemma}
\newtheorem{proposition}[theorem]{Proposition}

\newtheorem{definition}{Definition}
\newtheorem{example}{Example}
\newtheorem{remark}{Remark}

\newtheorem{observation}[theorem]{Observation}
\usepackage{csquotes}
\usepackage{listings}
\usepackage{verbatim}
\usepackage{fancyvrb}
\usepackage{caption}
\author{D. Abdullah}
\address{\textbf{ Duaa Abdullah:}
	The Moscow Institute of Physics and Technology}
\email{ abdulla.d@phystech.edu }
\thanks{ }
\author{J. Hamoud}
\address{\textbf{ Jasem Hamoud:}
	The Moscow Institute of Physics and Technology }
\email{ khamud@phystech.edu }
\thanks{ }
\title{Dynamic and Programmatic Analysis of Fibonacci Word Density}
\begin{document}

\begin{abstract}
Fibonacci word is the archetype of the Sturmian word, and it is one of the most studied of combinatorics on words.  We studied the properties of the Fibonacci word and found its density for limited value then by calculating the limit associated with the general relation of  $\frac{F(n)}{F(n+1)}$. We also generated the density function through the integral relation of $\int_a^b f(x) dx $ where $D(n)=\int_a^b f(x) dx$ and $f(x) = e^{-x/\tau} g(x)$.
 we presented the first study of the density of the Fibonacci word, in addition to some analysis through programming.
 \end{abstract}

\maketitle

\noindent\rule{12.7cm}{1.0pt}


\medskip


\maketitle
%%%%%%%%%%%%%%%%%%%%%%%%%%%%%%%%%%%%%%%%%%%%%%%%%%%%%%%%%%%%
\section{Introduction}\label{sec1}

Combinatorics on words delves into sequences of symbols (or "words") from a finite alphabet, there are alot of studies about combinatorics on words. The important basic concepts include as \textit{Non-Commutative Monoids} which it focus on sequences where the symbol order matters. \textit{Subword Complexity} which it analyses the number of unique subwords derivable from a word. and \textit{Infinite Words and Patterns} which it studies characteristics and regularities in infinitely repeated sequences.
It's all about understanding the complexities and patterns within symbol strings.\\
In 2014, Rigo, M. mention in ~\cite{1} to Word combinatorics  which has important applications in other branches of mathematics such that \textbf{Formal Language Theory} where it investigates the construction of languages and how word combinations shape language structure and parsing algorithms, \textbf{Automata Theory } which it examines the relationship between automata (abstract machines recognizing patterns) and words to understand word processing and recognition by computer models, and \textbf{Number Theory} where we have in number theory explore connections between arithmetic aspects and syntactical features of integer representations as finite words over a digit alphabet. Highlights include Cobham's theorem, linking growth rates of sequences described by various numeration systems (see ~\cite{1,2}).\\
In 2006, Glen, A., in ~\cite{3} in a PhD dissertation mentions to definition of Fibonacci word as an infinite binary sequence generated by a specific morphism. Starting with the initial conditions $F(0) = 0 $ and $ F(1) = 1 $, the sequence is constructed as follows:
  \[F(n) = F(n-1)F(n-2) \quad \text{ for } n \geq 2 \]
  The first few Fibonacci words are: 
  
     $ F(0) = 0,  F(1) = 1 ,  F(2) = 01, F(3) = 010, F(4) = 01001,  F(5) = 01001010, \dots $, and so on. 
     
 Also, in ~\cite{3} mention to Sturmian Words as Fibonacci words are a subset of Sturmian words, which are recognised for their complex combinatorial structure. They have qualities like a balance between the number of 0s and 1s, and they can be used to describe different phenomena in mathematics and computer science. 

In 2013, Ram\'irez, J.L., \& Rubiano, G.N. in ~\cite{4} mention to $k$-Fibonacci Words as The k-Fibonacci words are an extension of the Fibonacci word notion that generalises Fibonacci word features to higher dimensions. These words were investigated for their distinct curves and patterns. Also in 2014,  Catarino, P, mention to Binet\'s formula as $F_{k,n}=\frac{a_{1}^{n}-a_{2}^{n}}{a_1-a_2} \quad ;a_1>a_2$.
 
 In 2023, Rigo, M., Stipulanti, M., \& Whiteland, M.A. in ~\cite{6} mentioned to sturmian words and Thue-Morse sequence, which is $0\rightarrow 01, 1\rightarrow 10$, has $k$-binomial complexity for $k \geq 2$.  

 Recently, in 2024, Bacon, M.R., Cook, C.K., Fl\'orez, R., Higuita, R.A., Luca, F., \& Ram\'irez, J.L, in ~\cite{7} mention to $\sqrt{\lim_{n\rightarrow \infty}\frac{F_{n+2}}{F_n}} =\sqrt{\varphi}.$ in right tringle where an adjacent side of $\sqrt{F_{n+1}}$, an opposite side of $\sqrt{F_n}$ and a hypotenuse of $\sqrt{F_{n+2}}$.
%_------------------------------------------------------------
\subsection{Preliminaries}\label{sec1.1}

In this section we introduce an important concepts in combinatorics on words,especially on words that consider one of the most intuitive things we can define is the following. 
\begin{definition}
 Let be a finite set of elements of the series A as: $(a_1, a_2, ..., a_n)$ where $a_i \in A$, then we have All the words that we can mould from A will be denoted by the symbol $A^*$ , and we define for them a binary operation in the following form:
 \[
 (a_1, a_2,\dots, a_n) (b_1,b_2,\dots,b_n)= (a_1,a_2,\dots,a_n,b_1,b_2,\dots,b_n)
 \]
This binary process will allow us to form the following word:
$(a_1, a_2, ..., a_n)$ where $a_i \in A$ And all the words that we formed from A will be denoted by: $A^+=A^* -1$, and we call them the semi-group A (see ~\cite{8}).    
\end{definition}

We have the word $w=\left(a_1, a_2, \ldots, a_n\right); a_i \in A$ is made up of $n^{th}$ letters, we denote by $|\mathrm{w}|=\mathrm{n}$ the length of the word $w$ . If $\mathrm{w}$ does not contain letters, then its length will be $|w|:|\mathrm{w}|=0$. And if we have a function $w \mapsto|w|$ that is a morphism function of a free monoid $A^*$ in addition to a monoid $\mathbb{N}$ of positive integers, then we take a subset $\mathrm{B}$ of $\mathrm{A}$, and we denote by $|w|_B B$ the number of letters of $w$ that belong to $\mathrm{B}$. Therefore, $|w|=\sum_{a \in A}|w|_a$.
%----------------------------------------------------
\subsection{Infinite Square-Free Words}\label{sec1.2}
consider Thue- Morse is a basic base of Fibonacci word where the infinite word of Thue- Morse has square factors. In fact, the only forecourt-free words over two letters a and b are:$\{a,b,ab,aba,bab\}.$
On the opposite, there will be infinite square-free words over three letters. This will now be shown. 
\begin{definition}
 let $A=\{a,b\}$, and $B=\{a,b,c\}$. Define a morphism:  
 \[
 \delta:B^* \rightarrow  A^*; \quad \text{ where } \delta(c)=a,\delta (b)=ab,\delta(a)=abb
 \]
 For any infinite word b on $B$, 
$\delta(b)=\delta(b_{0})\delta(b_{1})\dots\delta(b_{n})$ (see ~\cite{8}).
\end{definition}

is an endless word on A that is clearly defined that begins with the letter $a$,  M. Lothaire in ~\cite{8} Consider, on the other hand, an infinite word a on $A$ that begins with $a$ and has no overlapping elements. Then $a$ can be considered as: 
\[
a=y_{0} y_{1}\dots y_{n} \dots, \quad \text{ where } y_{n}\in \{a,ab,abb\}=\delta(B)
\]
Since $bbb$ connects, each an in an is actually followed by no more than two b before a new $a$. The factorization is also special. Consequently, there is a special infinite word $b$ on $B$ such that $\delta(b)=a$ (see~\cite{8}).

Grimm, Uwe in 2001, in ~\cite{9} introduced some definintion are important about Infinite square-free words and bounds as: 
\begin{definition}[Square-Free Words~\cite{9}]
For any sequences that do not contain any consecutive repeated substrings called (squares).
\end{definition}
\begin{example}[Binary Square-Free Words]
Limited to short sequences like: $ a,b, ab, ba, aba, bab$.
\end{example}
\begin{definition}[Ternary Square-Free Words~\cite{10,11}]
 Can be infinitely long, and the number of such words grows exponentially with their length called Ternary Square-Free Words.
\end{definition}

Brandenburg in ~\cite{11} introduce a definition on bounds linked with Ternary Square-Free Words as 
\begin{proposition} [Bounds~\cite{11}]
    The number of ternary square-free words of length $n$ is bounded by $$6 \times 1.032^n \leq s(n) \leq 6 \times 1.379^n$$ as shown by Brandenburg in 1983.
\end{proposition}
  Zolotov, Boris,  Zolotov, Boris. in~\cite{12} shown in an example Extending a square-free word to avoid $ab$ such that in figure~\ref{fig1} (Furthermore, about trees see~\cite{ma01, ma02}), we can make linked with Ternary Square-Free Words, see following figure.
\begin{figure}[H]
    \centering
    \includegraphics[width=0.5\linewidth]{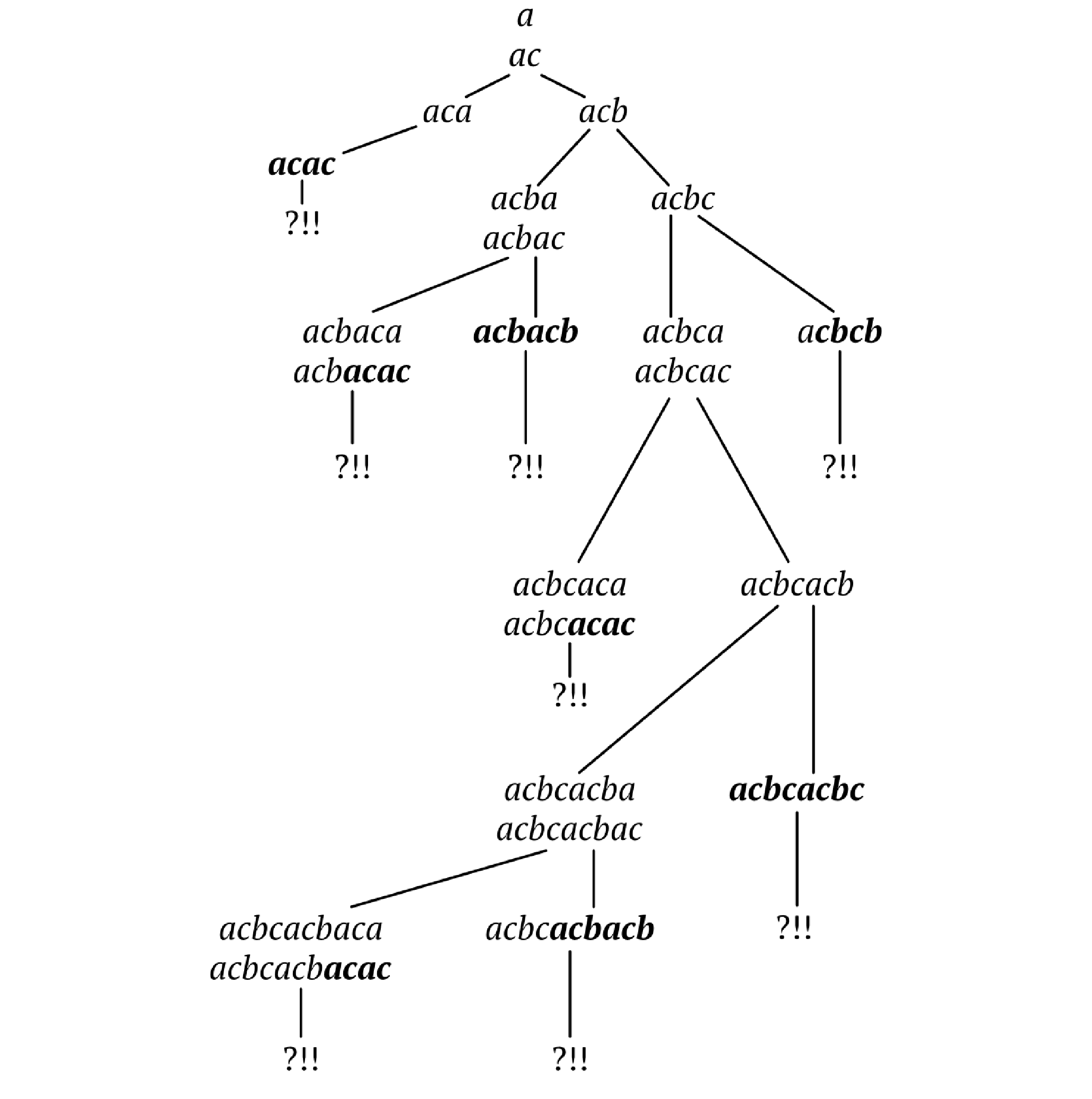}
    \caption{ Extending a square-free word to avoid $ab$. ~\cite{12}}
    \label{fig1}
\end{figure}
\begin{definition}[The Number of Occurrences]
    If we have the element $a \in A$ and $w \in A^*$ (we write $A^*$ it's the set of all words on A) intends the number of iterations $a$ in $w$ by $|w|_a$ and denotes the number of occurrences of $b$ in $w$ by $|w|_b$.
If $w=a b a a b$, we can put: $|w|_a=3, |w|_b=2.$ (see ~\cite{8}).
\end{definition}
%%%%%%%%%%%%%%%%%%%%%%%%%%%%%%%%%%%%%%%%%%%%%%%%%%%%
\section{Palindrome and scattered palindromic}	\label{sec2}

In this subsection, we will give a proof of the subsequence leading to a subword, which is an essential part when dealing with subwords in Fibonacci.
A \textbf{palindrome} is a sequence that reads the same forwards and backwards. This concept applies to various forms, including words, phrases, numbers, and even sequences in mathematics. the popular types of palindromes is: Character Palindromes, Word Palindromes, Phrase Palindromes, Number Palindromes, Date Palindromes,  Biological Palindromes.

In fact, what we are interested in Word Palindromes  and Number Palindromes in this paper which  define a\textbf{palindrome} in ~\cite{15} such a sequence $a_1, a_2, \ldots, a_n $ is considered a palindrome if: $ a_i = a_{n+1-i}; \quad 1\leq i \leq n$.  This means that the first element is equal to the last element, the second element is equal to the second last element, and so on also from ~\cite{15} we can insert code by python for know Whether it is the number is palindrome or no.
\begin{verbatim}
# Function to check if a number is a palindrome
def is_palindrome(num):
    num_str = str(num)
    # Check if the string is equal to its reverse
    return num_str == num_str[::-1]
# Main function
def main():
    while True:
        try:
         # Input from the user with prompt
         number = int(input("Enter a number: "))
            break  # Exit loop if input is valid
        except ValueError:
        print(" Please enter corecct integer.")
    # Check if the number is a palindrome
    if is_palindrome(number):
        print(f"{number} is a palindrome.")
    else:
        print(f"{number} is not a palindrome.")

if __name__ == "__main__":
    main()
\end{verbatim}
\begin{definition}
Strings or sequences that contain subsets of characters that can create palindromic structures are referred to be scattered palindromic. Sporadic palindromic structures are more flexible than standard palindromes, which have tight symmetry requirements. For scattered palindromes, we can define them in terms of their indices. A sequence $s_1, s_2, \ldots, s_k $ is scattered palindromic if there exists some mapping such that:
\[
s_i = s_j \quad \text{where } j = k + 1 - i \quad \text{for some } i,j.
\]
\end{definition}
\begin{example}
    Let be a word $w=aabb$, The scattered palindromic sub strings include: 
    \begin{enumerate}
        \item Single characters: $a, b$, 
        \item Pairs of identical characters: $aa, bb$, 
        \item Longer combinations that still maintain some palindromic properties: $aba$ (though not present in ``$aabb$'').
    \end{enumerate}
\end{example}
One way to find distributed palindromic sub strings in a given string is to search for every conceivable letter combination and see whether any of them can be rearranged to create palindromes. This method typically requires a lot of processing power because it creates every potential sub string and verifies each one's characteristics (see ~\cite{16}).

\begin{proposition}~\cite{17} \label{pro.5.3}
    Let $w$ be a finite word. Then, $P(w)\leq|w|$.
\end{proposition}

\begin{proposition}~\cite{18}
For a given finite word $w$, $P(w)\leq|w| \leq SP(w)$.
\end{proposition}

\begin{lemma}
	If we have $w$ any palindrome word, then it’s said that a word $v\in A^*$ is a subword of another word $x\in A^*$ if:	
	$v=a_1 a_2 \ldots a_n, a_i \in A, n \geq 0$ then we have also as: 
 $\exists \quad y_0, y_1, \ldots, \mathrm{y}_n \in \mathrm{A} ; \mathrm{x}=y_0 a_1 y_1 a_2 \ldots y_n a_n .$
	Therefore, we have: 
\begin{align*}
& |p(w)| \leq|s p(w)| ;\forall w \in \Sigma \\
& where \quad s p(w)=\sum_{t=1}^{|w|} s p_t(w) ; \quad \mathrm{t}: \text { length of word } w .\\
\end{align*}
\end{lemma}
\begin{proof}
Let be $v$ a sequence where $v=a_1 a_2 \ldots a_n, a_i \in A, n \geq 0$ and let be $x=y_0 a_1 y_1 a_2 \ldots y_n a_n .$ Since $v$ is a subsequence of $x$ as $v\subset x$, it’s a subword of $x$.
If we have a word $w=w^R$ is palindrome when be read both ways, so we see the first and second halves of the string are mirror images of each other.\\
And we refer to a dispersed $w$ as a scattered palindromic subword if it’s a palindrome.\\
The set of all non-empty palindromic subwords and dispersed palindromic subwords of $w$ are denoted, respectively, by $PAL(w)$ and $SPAL(w)$. Also, 
$$ P(w)=|P A L(w)| \text { and } S P(w)=|\operatorname{SPAL}(w)| \text {. }$$
Therefore we will be have many cases, as: \\
\textbf{Case 1}:  for finite words:\\
If we have the word $w=\mathrm{a b a a}$ , we can put: 
$$\begin{aligned}
		& p(w)=\{a, \mathrm{~b}, \mathrm{aa}, \mathrm{aba}\} \Rightarrow|p(w)|=4\\
		& s p(w)=\{\mathrm{a}, \mathrm{b}, \mathrm{a a},\mathrm{aba}, \mathrm{a a a}\} \Rightarrow|s p(w)|=5
\end{aligned}.$$
Also, for the word $w=\mathrm{a b a b}$ , we can write:
	$$\begin{aligned}
		& p(w)=\{{a}, {b}, {aa}, {bb}, {aba}, {bab}\}=s p(w) \\
		& \Rightarrow|p(w)|= |s p(w)|=6.
	\end{aligned}$$
\textbf{Case 2}: for the infinite word: \\
if we have the word as $w=\mathrm{abbaa...}$ , we can see:
	$$\begin{aligned}
		& p(w)=\{a, \mathrm{~b}, \mathrm{abba} ,\mathrm{aa}, \mathrm{bb},...\} \Rightarrow|p(w)|n\\
		& s p(w)=\{\mathrm{a}, \mathrm{b},\mathrm{abba}, \mathrm{a a} , \mathrm{bb}, \mathrm{aba},\mathrm{aaa},...\} \Rightarrow|s p(w)|\geq n
	\end{aligned}$$	
 from ~\ref{pro.5.3} we get on: 
 \begin{align*}
& |p(w)| \leq|s p(w)| ;\forall w \in \Sigma \\
& s p(w)=\sum_{t=1}^{|w|} s p_t(w) ; \quad \mathrm{t}: \text { length of word } w .\\
\end{align*}
as desire.
\end{proof}

\begin{lemma}~\cite{18}
At most $F_n$  extra scattered palindromic subwords can be created on the concatenation of a letter to a word of length $n-1$.
\end{lemma}
%%%%%%%%%%%%%%%%%%%%%%%%%%%%%%%%%%%%%%%%%%%%%%%%%%%%%%%%%%%
\section{Fibonacci word and The First Few Subwords} \label{sec3}

A particular infinite sequence in a two-letter alphabet makes up the infinite Fibonacci word.
Let $f_1=1, f_2=0, f_n=f_{n-1} f_{n-2}$ the concatenation of the two previous terms.\\
Additionally, the following morphism defines it: $\sigma: 0 \rightarrow 01,1 \rightarrow 0$
The Fibonacci words are given following:
\[
 f_1=1, f_2 =0, f_3 =01, f_4= 010, f_5=01001,		f_6 = 01001010, f_7 =01001011001001, \dots
\]	
\begin{proposition}[Infinite Fibonacci word~\cite{19}]
 The infinite Fibonacci word is the limited sequence of the unlimited sequence $f_\infty.$ then we have:  
$$\lim_{n\to\infty} f(n)=0100101001001010010100100101001001 \dots \label{re.1}$$    
\end{proposition}

It’s clear to see that $| fn |$  is the \(n^{th}\) Fibonacci number, which is called the ``Fibonacci word''. We have claimed that the Fibonacci word, which is a Sturmian word, is also pure morphic, exactly, Berstel and Karhum in ~\cite{20} mention to $f=w \frac{1}{\varphi^2}; \quad \text{where}\varphi=\frac{1+\sqrt{5}}{2}$ is the golden ratio.

The Fibonacci word is an infinite sequence of binary digits that can be constructed by starting with the two binary digits "0" and "1" and repeatedly appending the previous term to the current term. To find the subwords of the Fibonacci word, we can look at contiguous sequences of digits in the sequence. 
\begin{example}
The first few subwords of the Fibonacci word are:\\
$0, 1, 01, 10, 101, 010, 1010, 0101, 10101, 01010, 101010, ...$\\
In order to find the density of zero and one in the Fibonacci word for subwords with a length of less than thirty, we can start by looking at the first few subwords of the Fibonacci word.
\end{example}

We can see the subwords alternate between starting with 0 and starting with 1. This means that the density of zero and one in the Fibonacci word is equal, and is approximately $\frac{1}{2}$.
\begin{theorem}~\label{thm6.2}
  for the infinite Fibonacci word. we have $f_1=1, f_2=0, f_n=f_{n-1} f_{n-2}$, there are density of palindromes of Fibonacci word is: 
 \begin{table}[ht!]
	\centering
	\caption{density of palindromes in Fibonacci word for length $2,3,4$} \label{tab1}
	\begin{tabular}{|c|c|}
	\hline \hline 
\textbf{Length} &	\textbf{Density of Palindromes} 
\\ \hline \hline 
	2 &	$1/2$
\\ \hline
3	& $1/13, \quad 2/13$
 \\ \hline
 4 &	$1/21,\quad 2/21$
 \\ \hline 
	\end{tabular}
\end{table}
\end{theorem}
\begin{proof}
    To find the density of palindromes in the Fibonacci word, we can start by looking at the palindromes of length 2. These are: $00, 11, 01, 10.\quad$
We can see that the palindromes are the same form ~\ref{re.1} as: 
\begin{itemize}
    \item the subwords of length 2, except for the palindrome $"11"$ which is not a subword of length 2. The density of palindromes of length 2 in the Fibonacci word is equal to the density of subwords of length 2, which is also approximately $\frac{1}{2}$.
    \item For palindromes of length 3, we have: $ 000, 010, 101, 111, 001, 100.\quad$
     We can see that there are now six palindromes of length 3, and they do not all have the same density. In particular, the palindromes $000$ and $111$ each appear only once in the first 13 letters of the Fibonacci word, while the other four palindromes each appear twice. This means that the density of the palindromes $000$ and $111$ is approximately $\frac{1}{13}$, while the density of the other palindromes is approximately $\frac{2}{13}$.
    \item For palindromes of length $4$, we have: 
\[
0000, 0101, 1010, 1111, 0001, 1001, 0110, 1100.
\]
    We can see there are now eight palindromes of length 4, and they do not all have the same density. In particular, the palindromes $0000$ and $1111$ each appear only once in the first 21 letters of the Fibonacci word, while the other six palindromes each appear twice. This means that the density of the palindromes $0000$ and $1111$ is approximately $\frac{1}{21}$, while the density of the other palindromes is approximately $\frac{2}{21}$.
    \item  For palindromes of length 5 and greater, the number of palindromes and their densities become more complicated to calculate. However, it’s known that the density of palindromes in the Fibonacci word approaches zero as the length of the palindromes increases. This is because the Fibonacci word has a complicated and irregular structure, and as the length of the palindromes increases, it becomes more and more difficult for them to appear in the word with any significant frequency.
\end{itemize}
finally, we ca say the density of zero and one in the Fibonacci word for subwords with a length of less than thirty is approximately $\frac{1}{2}$. The density of palindromes in the Fibonacci word depends on the length of the palindromes, with shorter palindromes having a higher density than longer palindromes. However, as the length of the palindromes increases, their density approaches zero.
\end{proof}
%---------------------------------------------------
\subsection[density of length n.]{Calculate The Density of Fibonacci Word}\label{sec3.1}

The density of a subword $w$ in the Fibonacci word of length $n$ is defined as the number of occurrences of $w$ in the first $n$ terms of the Fibonacci word divided by $n$. In other words, it’s the proportion of the first $n$ digits of the Fibonacci word that are equal to $w$. To calculate the density of Fibonacci word of length $n$, we need to start by generating the first $n$ terms of the Fibonacci word, then we can count the number of times that the subword $w$ appears in the first $n$ terms of the Fibonacci word, and finally we divide by this number to obtain the density of each word. This means the density of $w$ is given by:
$$	 D(w)=\frac{C(w)}{\operatorname{n}}$$
\begin{example}
suppose we want to calculate the density of the subword $"101"$ in the first $1000$ terms of the Fibonacci word. 
We can generate the first 1000 terms of the Fibonacci word and count the number of occurrences of $"101"$ in this sequence.
Let's say we find that $"101"$ occurs $89$ times in the first $1000$ terms. 
Then the density of $"101"$ is:

	$$D("101")=\frac{89}{\operatorname{1000}}=0.089.$$
 \end{example}
\begin{proposition}
 approximately $8.9\%$ of the first $1000$ digits of the Fibonacci word are equal to $"101"$.
 \end{proposition}
 \begin{proof}
This is clear from the previous example and from studying the properties of the Fibonacci word (see ~\ref{re.1},~\cite{19}).
 \end{proof}
Research has shown that the density of Fibonacci words tends to a limit as $n$ increases, with longer words having a lower density of occurrence. Additionally, the density of each individual word appears to be related to the powers of the golden ratio, which plays a significant role in the structure of Fibonacci sequences.

The density of a word is defined as the limit of the ratio of the number of occurrences of a given letter in a word to the length of the word, at which the length of the word goes to infinity. For Fibonacci word, which is constructed by starting with the two symbols 0 and 1, and then repeatedly appending the previous two symbols to the end of the sequence.
We can see for example the Fibonacci sequence's first few terms are $A061107$ in OEIS \cite{21}. The dense Fibonacci word has a close relationship to the Fibonacci word fractal(see~\cite{22}), we should be mention in \ref{fin.1} introduced $\sqrt{\lim_{n\rightarrow \infty}\frac{F_{n+2}}{F_n}} =\sqrt{\varphi}$ in right triangle.

\begin{proposition}[Binet\'s formula \label{binfor}~\cite{5}]
for $n^{th}$ of $k$-Fibonacci numbers, we have: 
\[
F_{k,n}=\frac{a_{1}^{n}-a_{2}^{n}}{a_1-a_2} \quad ;a_1>a_2
\]
which $a_1,a_2$ are the roots equation $a^2-ka-1=0$ and $\{F_{k,n}\}_{n\in \mathbb{N}}$  is defined recurrently by: $F_{k,n+1}=kF_{k,n}+F_{k,n-1}\quad n\geq 1$ where: $F_{k,0}=0,F_{k,1}=1$, and we have: 
\[
\lim_{n\rightarrow \infty} \frac{F_{k,n}}{F_{k,n-1}}=a
\]
\end{proposition}
\begin{proposition}\label{projas}
 Let be $F_n$ Fibonacci numbers of oredr $n\geq 0$ and let be $\varphi=\frac{1+\sqrt{5}}{2}, b=\frac{1-\sqrt{5}}{2}$, which 
 $ F_n=\frac{\varphi^n-b^n}{\sqrt{5}}$
 then we have: 
 \[
 F_n \approx \frac{\varphi^n}{\sqrt{5}},
 \]
 then we have: 
 \[
 \lim_{n\rightarrow \infty} F_n= \infty.
 \]
\end{proposition}
\begin{proof}
from ~\ref{re.1} we have $$\lim_{n\to\infty} f(n)=0100101001001010010100100101001001 \dots,$$  by recognizing that denoted by $\varphi$ as the \textit{\textbf{golden ratio}}, which is approximately $\approx 1.618$. This value is greater than 1.
 from ~\ref{fig2} we have The subwords of the Fibonacci word for $n$ times so that as $n$ approaches infinity, the term $\left(\frac{1+\sqrt{5}}{2}\right)^n$ grows exponentially since it is raised to the power of $n$. then we have:
$$
\frac{\varphi^n}{\sqrt{5}} \to \infty \quad \text{as } n \to \infty.
$$

Thus, we can conclude that withe notic in ~\ref{re.1} as:

$$
\lim_{n \to \infty} F_n= \lim_{n \to \infty} \frac{\left(\frac{1+\sqrt{5}}{2}\right)^n}{\sqrt{5}} = \infty.
$$
\end{proof}
 in 2024 Bacon, M.R., Cook, C.K., Fl\'orez, R., Higuita, R.A., Luca, F., \& Ram\'irez, J.L, in ~\cite{7} introduced the following theorem, to calculation of the limit of Fibonacci numbers with respect to a right triangle, which is based on the fundamental principle of the golden ratio.
\begin{theorem}[\label{fin.1}~\cite{7}]
For $n\geq 1$, we have in a triangle exhibits an adjacent side of $\sqrt{F_{n+1}}$, an opposite side of $\sqrt{F_n}$ and a hypotenuse of $\sqrt{F_{n+2}}$,  we have: 
\[
\sqrt{\lim_{n\rightarrow \infty}\frac{F_{n+2}}{F_n}} =\sqrt{\varphi}.
\]
\end{theorem}

 \begin{theorem}\label{mainre}
The density of Fibonacci word a particular letter in the word using the following formula:
	$$\lim _{n\rightarrow \infty}D(n)=\lim _{n\rightarrow \infty}(\frac{F(n)}{\operatorname{F(n+1)}})=\varphi-1$$
where $\varphi$ is the golden ratio.
\end{theorem}
\begin{proof}
The density of a sequence is a measure of how many times a particular symbol appears in the sequence, relative to the total length of the sequence. In the case of Fibonacci word, the density of the symbol 1 can be shown to be the golden ratio, which is approximately $1.61803398875..$ as in~\ref{projas}.

To see why this is the case, let's consider the process of constructing Fibonacci word. We start with two symbols, 0 and 1, and at each step, we add the previous two symbols to the end of the sequence. This means from theorem ~\ref{thm6.2} that the number of 1's in the sequence is equal to the number of times we add the symbol 1 to the sequence (see figure ~\ref{fig2}), then we have: 
\begin{itemize}
    \item At the first step, we add a single 1 to the sequence, so the density of 1's is \(\frac{1}{2}\).
    \item At the second step, we add another 1 to the sequence, so the density of 1's becomes \(\frac{2}{3}\).
    \item At the third step, we add two 1's to the sequence, so the density of 1's becomes \(\frac{3}{5}\).
\end{itemize}
At each subsequent step, the density of 1's approaches the golden ratio $\varphi$ as the length of the sequence grows larger and larger.

Otherwise, we can be proven mathematically using the fact Fibonacci word can be generated using a substitution rule, where each occurrence of the letter $"0"$ is replaced with the string $"01"$ and each occurrence of the letter "1" is replaced with the string "0".

Using this substitution rule, we can come to a formula for the \(n^{th}\) letter of the Fibonacci word in terms of the binary expansion of $n$. Specifically, the \(n^{th}\) letter of the Fibonacci word is equal to 1 if the number of ones in the binary expansion of $n$ is one more than a multiple of 3, and 0 otherwise(see ~\ref{stu1}).

This can be shown mathematically using the fact that the ratio of consecutive Fibonacci numbers like in ~\ref{re.1} shown approaches the golden ratio $\varphi$ as the numbers get larger. Specifically, let F(n) denote the \(n^{th}\) Fibonacci number (with $F(1)=0$ and $F(2)=1$). 
	
Using the formula for the \(n^{th}\) Fibonacci number:
$$F(n)=\frac{\varphi^n-(1-\varphi)^n}{\operatorname{\sqrt{5}}}$$
So the density of 1's in the first n terms of the Fibonacci word is given by:
$$ D(n)=\frac{F(n)}{F(n+1)}$$
	
This can be done by considering the density of the ones in the first $2^n$ letters of the Fibonacci word and using the fact that the number of ones in the binary expansion of the integers from 0 to $2^n$-$1$ is approximately:
\[
\frac{2^n}{2^n\varphi-1}
\]
More specifically, we can show that the density of ones in the first $2^n$ letters of the Fibonacci word is approximately: $\frac{2^n}{\varphi}$, which approaches $\varphi$ as n goes to infinity. This implies that the density of ones in the entire Fibonacci word is also equal to $\varphi$.

Therefore, we can see the following formula:
	$$\lim _{n\rightarrow \infty}D(n)=\lim _{n\rightarrow \infty}(\frac{F(n)}{\operatorname{F(n+1)}}).$$
This means that as n gets larger and larger, the density of 1's in the first n terms of the Fibonacci word approaches the golden ratio. In other words, the Fibonacci word has a density of 1's that is asymptotically equal to the golden ratio $\varphi$.\\
The density of the symbol 0 in the Fibonacci word is also equal to the golden ratio , since the sequence is symmetric with respect to the two symbols. Therefore, the Fibonacci word has a uniform density of \(\frac{1+\varphi}{2}\) for both 0's and 1's.\\
Therefore, we can conclude that the density of the Fibonacci word is equal to $\varphi-1$. So we can write:    
	$$\lim _{n\rightarrow \infty}D(n)=\lim _{n\rightarrow \infty}(\frac{F(n)}{\operatorname{F(n+1)}})=\varphi-1.$$
 We have:
	$$F(\mathrm{n})=\frac{\varphi^n-(1-\varphi)^n}{\sqrt{5}}, F(n+1)=\frac{\varphi^{n+1}-(1-\varphi)^{n+1}}{\sqrt{5}}$$
	Thus:
	$$F(n+1)=\frac{\varphi \varphi^n-(1-\varphi)(1-\varphi)^n}{\sqrt{5}}$$
	So that, we have in figure ~\ref{jfig.1} density for n=10 as example in many cases as in figure, so that: 
 \[
\frac{F(\mathrm{n})}{F(n+1)}=\frac{\frac{\varphi^n-(1-\varphi)^n}{\sqrt{5}}}{\frac{\varphi \varphi^n-(1-\varphi)(1-\varphi)^n}{\sqrt{5}}} = \frac{\varphi^n-(1-\varphi)^n}{\varphi \varphi^n-(1-\varphi)(1-\varphi)^n}
 \]
 Therefore we notice that in the figure bellow as: 
\begin{figure}[H]
    \centering
    \includegraphics[width=0.5\linewidth]{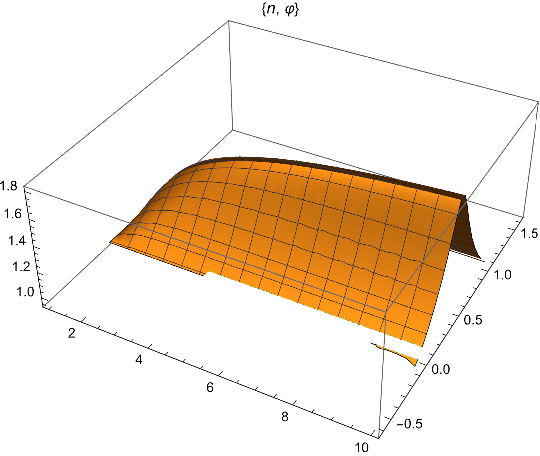}
    \caption{Density of Fibonacci word with Mathematica of length $n=10$.}
    \label{jfig.1}
\end{figure}
Also, we have the figure~\ref{jfig.2} as:
\begin{align*}
     & \lim _{n \rightarrow \infty} \frac{F(\mathrm{n})}{F(n+1)}=\\
     &  \lim _{n \rightarrow \infty} \frac{2\left(\left(\frac{1}{2}(1+\sqrt{5})\right)^n-e^{n(\log (\varphi-1)+i \pi)}\right)}{2^{-n}(1+\sqrt{5})^{n+1}+(\sqrt{5}-1) e^{n(\log (\varphi-1)+i \pi)}}=\\
     &=0.618033 \ldots \\
 \end{align*}
\begin{figure}[H]
    \centering
    \includegraphics[width=0.5\linewidth]{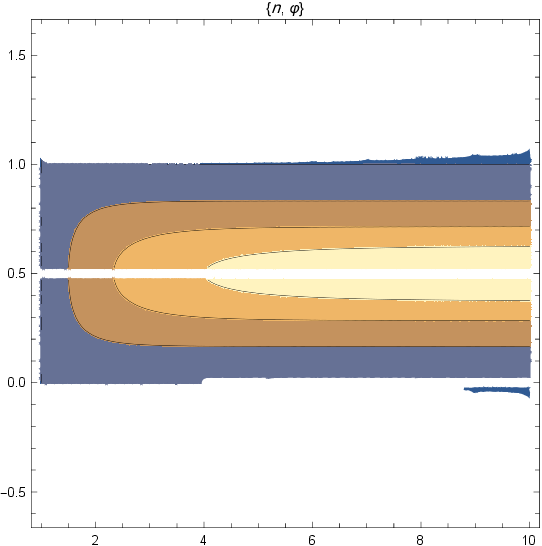}
    \caption{Cross section of Density of Fibonacci of length $n=10$.}
    \label{jfig.2}
\end{figure}
 Finally we have: 
 \[
 \lim _{n \rightarrow \infty} \frac{F(\mathrm{n})}{F(n+1)}=\varphi-1 \approx 0.618033 
 \]
 Therefore $\lim _{n\rightarrow \infty}D(n)=\lim _{n\rightarrow \infty}\frac{F(n)}{\operatorname{F(n+1)}}$ as n approaches infinity is approximately $\varphi-1\approx 0.61803398874...$,which $\varphi$   it’s the golden ratio.
\end{proof}
\begin{remark}
 We can more clearly express the graphs we referred in ~\ref{jfig.1},\ref{jfig.2} to in the previous proof as follows: We extend Figure~\ref{jfig.2} to:
 \begin{figure}[H]
    \centering
    \includegraphics[width=0.5\linewidth]{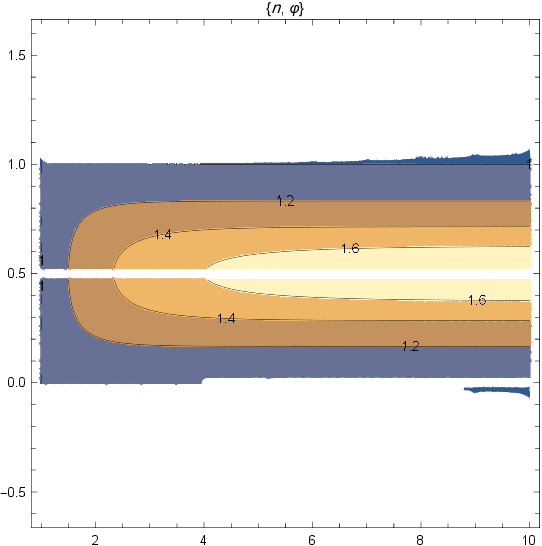}
    \caption{Extend Density of Fibonacci~\ref{jfig.2}.}
    \label{jfig.3}
\end{figure}
Also, Stretch Figure~\ref{jfig.3} to:  
\begin{figure}[H]
    \centering
    \includegraphics[width=0.5\linewidth]{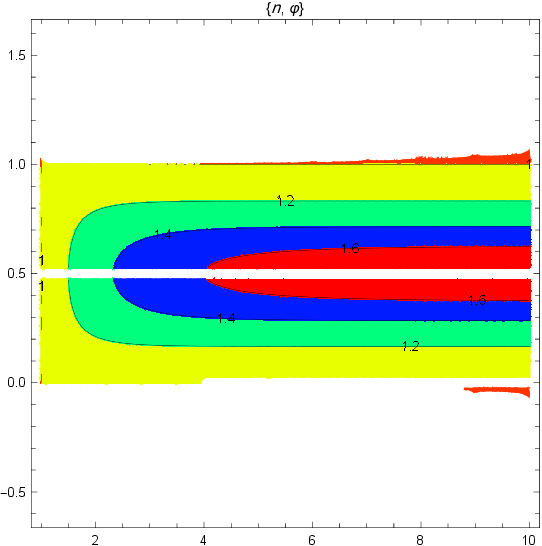}
    \caption{ Hue of Density of Fibonacci of length $n=10$.}
    \label{jfig.4}
\end{figure}
\end{remark}
The Fibonacci word density can be generated from ~\ref{binfor}, and ~\ref{mainre} by integrating as $\int_a^b f(x) dx\quad \text{where } a>b \text{ and } a,b \in \mathbb{N}  $ in next lemma we show that. The integral represents an approximation derived from the calculation of the limit of the follower that we referred to in the previous theorem~\ref{mainre}. 
\begin{lemma}[\label{intger}]
Let be $k\geq 0$, and a function $f$ define as: $f: \mathbb{N} \rightarrow \mathbb{R}, \text{ where } f(x)={e^{{-x}(1+1/\tau)}}x^{k-1}$, then dense of Fibonacci words can be generated as:
$$ \lim_{n\rightarrow \infty} D(n) =\lim_{n\rightarrow \infty} \int_a^b f(x) dx \quad \text{where } a>b  \text{ and } a,b \in \mathbb{N}  $$
where $D(n)=\int_a^b f(x) dx$ and $f(x) = e^{-x/\tau} g(x) $, and $ g(x) = x^{k-1} e^{-x} $, for some appropriate choices of $ a, b, k, \tau $, Thus, we have: 
$$ \lim_{n\rightarrow \infty} D(n) =\lim_{n\rightarrow \infty} \int_a^b {e^{{-x}(1+1/\tau)}}x^{k-1} .dx$$
\end{lemma}
\begin{proof}
by using Binet\'s formula from ~\ref{binfor} we have:
$ F(n) = \frac{\varphi^n - (1 - \varphi)^n}{\sqrt{5}}, $
where $ \varphi = \frac{1 + \sqrt{5}}{2} $ is the golden ratio.
we can express $ D(n) $ as we refer in ~\ref{mainre} as :
$$ D(n) = \frac{\varphi^n - (1 - \varphi)^n}{\varphi^{n+1} - (1 - \varphi)^{n+1}}. $$

Thus, we get on the term $ (1 - \varphi)^n $ where $ n \to \infty $  approaches zero because $ |1 - \varphi| < 1 $.  from ~\ref{mainre},  we can simplify with $a>b \text{ and } a,b \in \mathbb{N} $ then (see ~\ref{jfig.4}) we have:

$$\lim_{n\rightarrow\infty} D(n) = \varphi-1 \approx 0.61803398874... $$

by utilising the continuous analogues of Fibonacci numbers and their attributes. The limit converges to a golden ratio-related constant, which can be represented by integrals incorporating alternative continuous approximations of Fibonacci growth or exponential decay functions as: 
\[
D(n) = \int_a^b f(x)\quad \text{ then, } \lim_{n\rightarrow \infty} D(n) =\lim_{n\rightarrow \infty}\int_a^b {e^{{-x}(1+1/\tau)}}x^{k-1} .dx
\]

we have $a>b, \{a,b\}\in \mathbb{N}$, As long as this condition is fulfilled, by the properties of the Fibonacci word, we get the value in \ref{mainre}, While a specific integral formula may vary depending on context and application, it will generally reflect relationships between exponential functions and power laws that resemble the growth rates of Fibonacci numbers as $ n \to \infty $.
\end{proof}
\begin{lemma}
    The ratio of consecutive Fibonacci numbers converges to the golden ratio:
     $$\lim_{n \to \infty} \frac{F(n)}{F(n-1)} = \varphi$$
\end{lemma}
\begin{proof}
we have: $\lim_{n \to \infty} D(n) = \varphi - 1 $, the limit states that as $ n $ approaches infinity, $ D(n) $ approaches $ \varphi - 1 .$, from ~\ref{projas} we have $\lim_{n\rightarrow \infty} F_n= \infty.$, we know from ~\cite{3} $F(n) = F(n-1) + F(n-2)$, then we have: 
$$ D(n) = \frac{F(n)}{F(n+1)} = \frac{F(n)}{F(n) + F(n-1)} $$
   As $ n \to \infty $, we can consider the relationship between these ratios.
   Thus from ~\ref{mainre} for any  Fibonacci numbers we have : 
   $$\lim_{n \to \infty} \frac{F(n)}{F(n-1)} = \varphi$$
\end{proof}
\begin{observation}\label{lemfin}
by using Lemma \ref{intger} and if we have: $\lim _{n\rightarrow \infty}D(n)=\lim _{n\rightarrow \infty}(\frac{F(n)}{\operatorname{F(n+1)}})=\varphi-1$, then we can therefore make an assumption based on the properties of power series as follows:
 \[
  D(n)\approx e^{-\sum_{k=1}^{n} (\varphi - 1)}.
 \]
\end{observation}
%\begin{proof}
In fact, this is an approximation because $D(n) = \frac{F(n)}{F(n+1)} $, then $\lim _{n\rightarrow \infty}D(n)=\varphi - 1$, and we have $D(n) = e^{-\lim_{n\to\infty} (\varphi - 1)n}$.
However, since $ D(n) \to \varphi - 1 $, we need to adjust this to reflect the correct product form. So to express the limit of your sequence as a product, we can write:
$ D(n) = e^{-\sum_{k=1}^{n} (\varphi - 1)} $
Thus, we conclude:
$$ \lim_{n\to\infty} D(n) = e^{-\infty} = 0 \neq \varphi - 1$$
So that, this is an approximation. Thus, the assumption we have made is not correct assumption based on the integral we have presented, we can notice that in this figure beside on \ref{mainre} as:
\begin{figure}[H]
    \centering
    \includegraphics[width=0.5\linewidth]{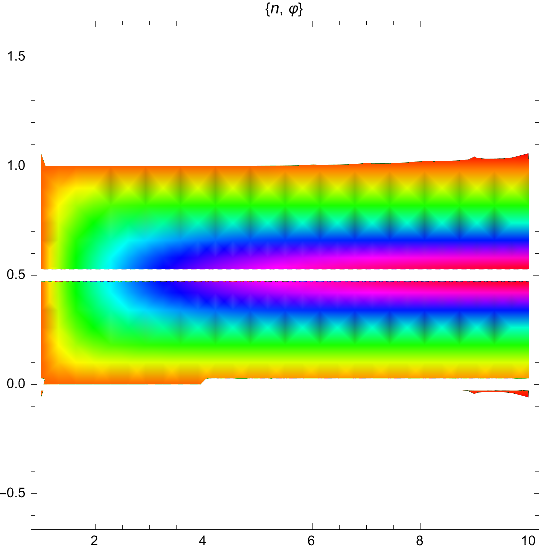}
    \caption{Density of Fibonacci word of length $n=10$.}
    \label{jfig.5}
\end{figure}
%\end{proof}
%---------------------------by programming ----------------
\subsection{Evaluate Fibonacci word density programmatic and dynamically.}\label{sec3.2}

The density of Fibonacci word for the first one 100 digits by Python programming we can write the following:
\begin{verbatim}
def fibonacci_word(n):
   if n == 0:
   return ""
   elif n == 1:
   return "1"
   elif n == 2:
   return "10"
   else:
   return fibonacci_word(n-1) + fibonacci_word(n-2)
   fib_word = fibonacci_word(22) 
#get the first 100 digits of the Fibonacci 
   num_ones = fib_word.count("1")
   num_zeros = fib_word.count("0")
import math
   phi = (1 + math.sqrt(5)) / 2
def fibonacci(n):
   return int((phi**n - (1-phi)**n) / math.sqrt(5))
def limit_ratio(n):
   return fibonacci(n) / fibonacci(n+1)
#compute the limit of the ratio as n approaches infinity
  n = 1000 #any large number can be used here
while True:
  ratio = limit_ratio(n)
  next_ratio = limit_ratio(n+1)
  if abs(ratio - next_ratio) < 1e-10:
break
  n += 1
# print the result
print("for the first 100 digits:", fib_word)
print("Number of ones:", num_ones)
print("Number of zeros:", num_zeros)
print("density = as n approaches infinity: 
{:.10f}",format(ratio))
\end{verbatim}
The output of this program will be:

for the first 100 digits: 101101011011010110101101101011$\dots$

Number of ones: $17711$

Number of zeros: $10946$

density = as n approaches infinity: 0.6180339887498$\dots$

we found in ~\ref{jfig.1}, density of Fibonacci word with Mathematica of length $n = 100$, in next figure we have the length for $n=100$ as: 
\begin{figure}[H]
    \centering
    \includegraphics[width=0.5\linewidth]{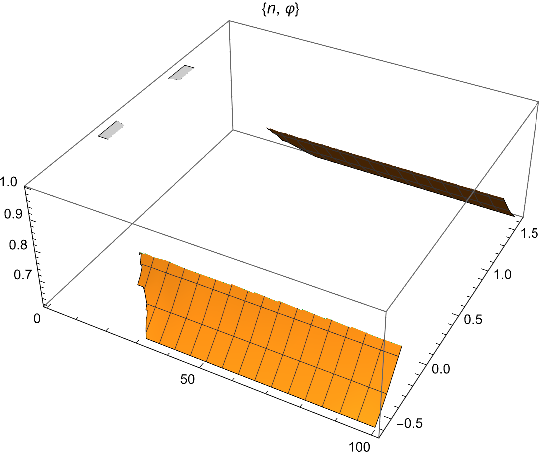}
    \caption{Density of Fibonacci word with length $n=100$.}
    \label{jfig.6}
\end{figure}

We can see from the figure, which is the visualisation of the Fibonacci word density relationship, that as the length of the word increases as in this figure from ~\ref{jfig.2} as: 
\begin{figure}[H]
    \centering
    \includegraphics[width=0.5\linewidth]{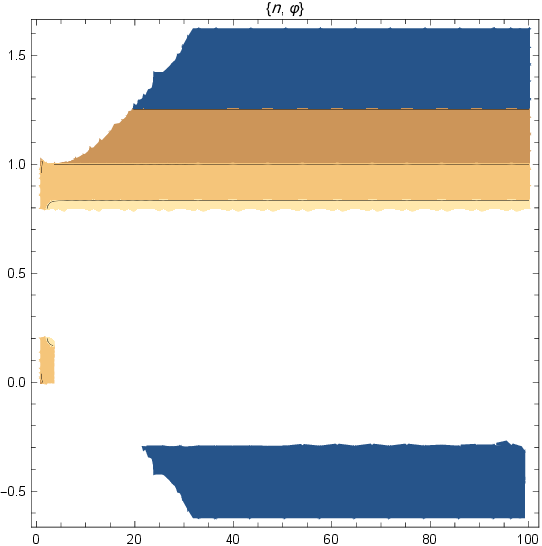}
    \caption{Extend of length $n=100$.}
    \label{jfig.7}
\end{figure}
the density decreases and this is what we found with us through lemma~\ref{lemfin}, where we got a limit equal to 0, which is an unacceptable value, and also in~\ref{projas} we have a value equal to infinity, which is an unacceptable value because we are dealing with finite words and this is what the graph shows us for a length equal to 10 and a length equal to 100, we explain in next figure as:
\begin{figure}[H]
    \centering
    \includegraphics[width=0.5\linewidth]{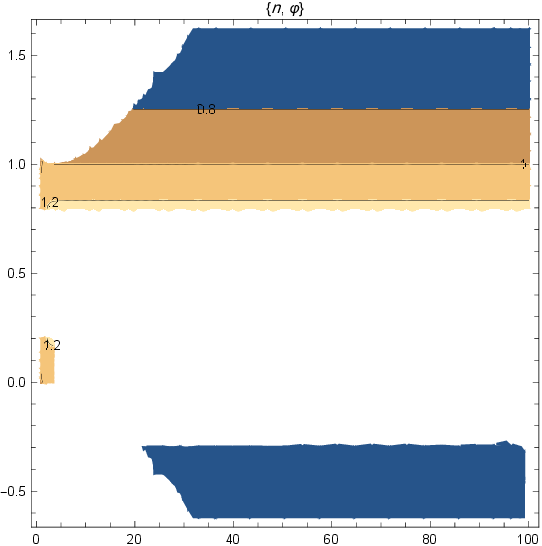}
    \caption{Extend Density of Fibonacci~\ref{jfig.7} with $n=100$.}
    \label{jfig.8}
\end{figure}
Finally we have the Hue figure as: 
\begin{figure}[H]
    \centering
    \includegraphics[width=0.5\linewidth]{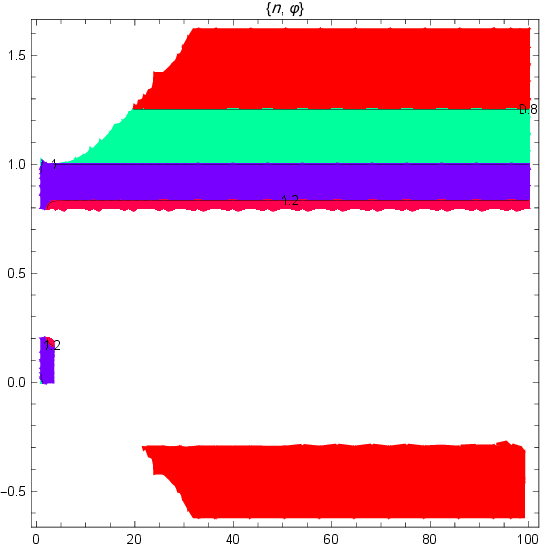}
    \caption{Hue of Density of Fibonacci word with $n=100$.}
    \label{jfig.9}
\end{figure}
%%%%%%%%%%%%%%%%%%%%%%%%%%%%%%%%%%%%%%%%%%%%%%%%%%%%
\section{Catalen Numbers in  Density of Fibonacci Word}\label{sec4}

Catalan numbers are a sequence composed of organic numbers that have been used in a variety of combinatorial mathematics tasks. They are best known for counting the number of valid parenthetical expressions, grid paths, and other structures. The sequence starts with one, two, five, fourteen, and so on. These numbers, which may be generated via recursive formulas, offer intriguing features and applications in a variety of mathematical disciplines.

Catalan numerals are employed to count valid brackets. For example, with n pairs of brackets, the nth Catalan number indicates the number of appropriate groupings (see~\cite{23}).
They also appear in permutations and layouts, such as when counting specific permutations utilising Catalan number squares (see~\cite{24}. Given that Catalan numbers are most commonly linked with combinatorial problems, their use in other mathematical fields, such as algebraic geometry and number theory, demonstrates its variety and depth. This broad applicability emphasises the significance of Catalan numbers in both fundamental and applied mathematics.
In~\cite{24} define catalen number as:
\[
\mathcal{C}_n=\frac{1}{n+1}\binom{2n}{n}
\]
\begin{proposition}
Fibonacci word of catalen number has defined as: 
\[
F(\mathcal{C}_n)=F(\mathcal{C}_n-1)F(\mathcal{C}_n-2).
\]
where fibonacci word of order $n\geq 3$.
\end{proposition}
\begin{proof}
We have $\binom{2n}{n}=\frac{(2n)!}{(n!)^2}$ depend on proposition as $\binom{x}{y}=\binom{x}{x-y}$, then we have: 
\[
F(\mathcal{C}_n)=F(\frac{1}{n+1}.\frac{(2n)!}{(n!)^2}-1)F(\frac{1}{n+1}.\frac{(2n)!}{(n!)^2}-2)
\]
this will be correct for $n\geq 3$ so that we have for
$\frac{1}{n+1}.\frac{(2n)!}{(n!)^2}-1$ following table as: 

\begin{table}[h]
    \begin{tabular}{|c|c|}
    \hline
       $n$  & $\frac{1}{n+1}.\frac{(2n)!}{(n!)^2}-1$ \\ \hline 
        1 & $-1/2$ \\ \hline
        2& $-1/3$ \\ \hline 
        3 & $1/4$  \\ \hline
        4 & $9/5$ \\ \hline 
    \end{tabular}
    \caption{Correct value of $n$}
    \label{tab 1}
\end{table}

\textbf{Case 1: } Let be take: $\frac{1}{n+1}.\frac{(2n)!}{(n!)^2}-1$, \quad in this case we can write it as: 
\[
\frac{\binom{2n}{n}-n^2-2n-1}{(n+1)^2}=-\frac{\binom{2n}{n}+n^2+2n+1}{(n+1)^2}
\]
by using gamma function $\Gamma$ we can write: 
\[
\frac{2^{2n}\Gamma(n+\frac{1}{2})}{\sqrt{\pi}(n+1)^2\Gamma(n+1)}-1 = \frac{4^n\Gamma(n+\frac{1}{2})}{\sqrt{\pi}(n+1)^2\Gamma(n+1)}-1 = \frac{2^{2n+1}(\frac{1}{2}(2n+1))!}{\sqrt{\pi}(n+1)^2(2n+1)(n!)}-1
\]
So that, we have: 
\[
-\frac{\sqrt{\pi} (n+1)^2(2n+1)(n!)-2^{2n+1}(\frac{1}{2}(2n+1))!}{\sqrt{\pi}(n+1)^2(2n+1)(n!)}.
\]
\textbf{Case 2} Let be take $\frac{1}{n+1}.\frac{(2n)!}{(n!)^2}-2$, \quad in this case we can write it  immediately by using gamma function $\Gamma$ as: 
\[
\frac{2^{2n}\Gamma(n+\frac{1}{2})}{\sqrt{\pi}\Gamma(n+2)}-2=\frac{2^{2n}\Gamma(n+\frac{1}{2})-2\sqrt{\pi}\Gamma{n+2}}{\sqrt{\pi}\Gamma(n+2)}
\]
So that, we have: 
\[
\frac{-2\sqrt{\pi} (n+1)(2n+1)n!-2^{2n+1}(\frac{1}{2}(2n+1))!}{\sqrt{\pi} (n+1)(2n+1)n!}.
\]
Finally, we have: 

\begin{align*}
   & F(\frac{1}{n+1}\binom{2n}{n})= \\
  & F(-\frac{\sqrt{\pi} (n+1)^2(2n+1)(n!)-2^{2n+1}(\frac{1}{2}(2n+1))!}{\sqrt{\pi}(n+1)^2(2n+1)(n!)}).\\
  &F(\frac{-2\sqrt{\pi} (n+1)(2n+1)n!-2^{2n+1}(\frac{1}{2}(2n+1))!}{\sqrt{\pi} (n+1)(2n+1)n!}).
\end{align*}
where $n\geq 3$. In fact, this is identical to the traditional definition of Fibonacci word- which is identical to Morse sequence - and is identical to the original definition of Fibonacci numbers.

As desire.
\end{proof}
\begin{theorem}
Let be $n\geq 1$ integer number, density of fibonacci word by using catalen number $\mathcal{C}_n$ is: 
\[
D(F(\mathcal{C}_n))= \lim _{n\rightarrow \infty} \frac{F(\mathcal{C}_n+1)}{F(\mathcal{C}_n)}=1.
\]
\end{theorem}
\begin{proof}
In order to calculate the density, since we have expressed it as an increasing limit of 1, we need to use a Fibonacci term and let this term be a term imposed on this mathematical basis, that is, we express the Fibonacci term in the form of a mathematical term given by a specific formula depend on the theorem of density is: density of Fibonacci word a particular letter in the word using the following formula:
	$$\lim _{n\rightarrow \infty}D(n)=\lim _{n\rightarrow \infty}(\frac{F(n)}{\operatorname{F(n+1)}})=\varphi-1$$
where $\varphi$ is the golden ratio. Let be suppose the function is: $F(n)= \frac{(n+1)(n!)^2+(2n)!}{(2n)!}$, then we can rewrite the limit as:

$$
L = \lim_{n \to \infty} \left( \frac{(n+1)(n!)^2}{(2n)!} + 1 \right).
$$
 Using the general definition of the factorial's function, which gives an approximation, we can write by using Stirling's 
 such that $ n! \sim \sqrt{2\pi n} \left( \frac{n}{e} \right)^n $, we can approximate $(2n)!$ as $(2n)! \sim \sqrt{4\pi n} \left( \frac{2n}{e} \right)^{2n}.$ Then we have: 
$$
L = \lim_{n \to \infty} \left( \frac{(n+1)(\sqrt{2\pi n} (n/e)^n)^2}{\sqrt{4\pi n} (2n/e)^{2n}} + 1 \right) = \lim_{n \to \infty} \left( \frac{(n+1)(2\pi n) (n^2/e^{2n})}{\sqrt{4\pi n} (4^n n^{2n}/e^{2n})} + 1 \right).
$$

Now, simplifying the fraction in figure~\ref{fja1} as:
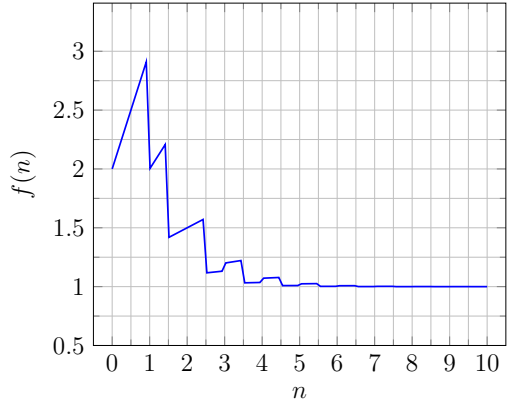
\begin{figure}[H]
    \centering
\begin{tikzpicture}[scale=.8]
    \begin{axis}[
        xlabel={$n$},
        ylabel={$f(n)$},
        domain=0:10,
        samples=100,
        grid=both,
        minor tick num=1,
        enlarge x limits={abs=0.5},
        enlarge y limits={abs=0.5},
        xtick={0,1,...,10},
        ytick={0,0.5,...,3},
    ]
    % Define the function to be plotted
    \addplot[blue, thick] {((x+1)*(factorial(x)^2) + factorial(2*x))/(factorial(2*x))};
    \end{axis}
\end{tikzpicture}
\caption{The function $f(n)$ on domain $[0,10]$.} \label{fja1}
\end{figure}
%Then we have
$$
 \lim_{n \to \infty} \left( \frac{(n+1)(2\pi n)}{\sqrt{4\pi n} (4^n)} + 1 \right) = \frac{(n+1)(2\pi n)}{\sqrt{4\pi n}(4^n)} = 0.
$$

Thus, we find that:
$$
\lim_{n \to \infty} \left( \frac{(n+1)(n!)^2 + (2n)!}{(2n)!} \right) = 1.
$$
Thus, we have obtained the limit of the dependent variable that we assumed to be equal to 1. We can adopt this result to form our hypothesis to obtain the density of the Fibonacci word using Catalan numbers to be equal to 1, which is the desired value.

\end{proof}
%-----------------------------------------------------------
\subsection{Fuzzy of Fibbonacci Word}\label{sec4.1}
Applying fuzzy logic into this concept may imply incorporating ambiguity or degrees of membership in the strings formed. For example, we could design a fuzzy Fibonacci word by allowing specific characters to appear in varied degrees of frequency.\\

\textit{\textbf{Example:}}
Let’s create a fuzzy version of the Fibonacci word by assigning fuzzy values to characters:
   \textbf{Define Characters with Fuzzy Membership:}
  
       $$ ``a'' could have a membership value of 0.8 (strong presence).$$
        
        $$  ``b'' could have a membership value of 0.5 (moderate presence).$$
   \textbf{Constructing the Fuzzy Word:}
       Start with the base cases:
            $$ F(0) = ``b'' (0.5)$$ 
            $$  F(1) = ``a''  (0.8)$$
         Construct further:
           $$ F(2) = ``ab''  (combine memberships: min(0.5, 0.8))$$
           $$ F(3) = ``aba''$$
            $$ F(4) = ``abaab'' $$      

You can represent each word with its fuzzy membership values:
      $$ F(2) = ("a", 0.8), ("b", 0.5) $$
      $$ F(3) = ("a", 0.8), ("b", 0.5), ("a", 0.8) $$

%%%%%%%%%%%%%%%%%%%%%%%%%%%%%%%%%%%%%%%%%%%%%%%%%%%%%%%%%%%%%
\section{Conclusion}\label{sec5}

Combinatorics on words plays a crucial part in many branches of mathematics. the theory of infinite binary sequences called Sturmian words, is the most well-known of its branches and is intriguing in many ways.

We referred to a relationship for any word in all non-empty palindromic subwords and dispersed palindromic subwords. 

The Fibonacci word deserves, is that its properties be studied further apart from its particular beauty, which is frequently used and has significance in the theory of words and combinatorial pattern matching, is the most well-known illustration of a Sturmian word.

The first n terms of the Fibonacci word have a density of 1's that approaches the golden ratio, which is $$\lim _{n\rightarrow \infty}D(n)=\lim _{n\rightarrow \infty}(\frac{F(n)}{\operatorname{F(n+1)}})=\varphi-1 =0.61803398875..$$\\
 This indicates that as the sequence increases, the ratio of 1s in the sequence approaches the golden ratio.
 
The number of subwords of length $n + 1$ increases by 1 for both letters, 0 and 1.

As a result, the density of both 0's and 1's in the Fibonacci word is $\frac{1+\varphi}{\operatorname{2}}$.\\

By Python programming we write a program and founded the natural density for fibonacci word for the first one 100 digits and after that we drew the graph by Mathematica for n=10,100,..., In addition, we can extend the histogram in~\ref{jfig.9} to: 
\begin{figure}[H]
    \centering
    \includegraphics[width=0.5\linewidth]{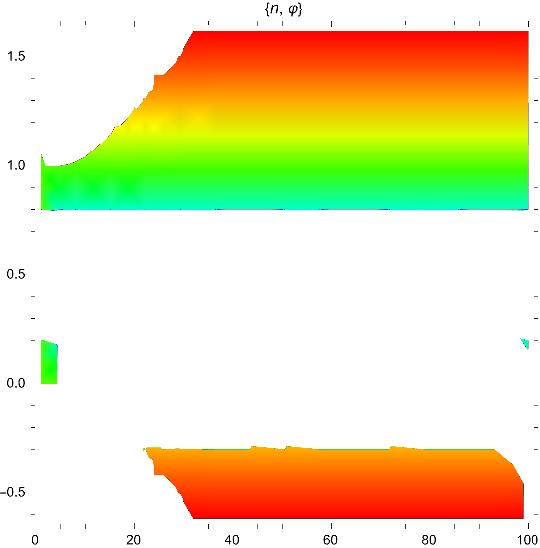}
    \caption{finally extend density of fibonacci word with length $n=10$.}
    \label{jfig.10}
\end{figure}
Therefore we get on a clear overview of the density of the Fibonacci word for length $n=10,100$, and by programme to $\infty$.

In general, the Fibonacci word has received a great deal of attention due to its importance in discrete mathematics as well as its useful uses in computer graphics, theoretical physics, and molecular biology.

\section*{Acknowledgements}

Thank you to prof. Alexei Kanel-Belov for their thoughtful feedback and consistent support throughout these years, which elevated the methodological approaches in this work.

\section*{Declarations} 

\begin{itemize}
%\item Funding  Not applicable
\item Conflict of interest/Competing interests: No Conflicts: ``All authors declare that they have no conflicts of interest.'' 
\item Ethics approval and consent to participate: This study did not require ethics approval as it did not involve human participants.
\item Author contribution: All Authors Reviewed and Approved the Final Version of the Manuscript.
\end{itemize}

%% if required, the content of .bbl file can be included here once bbl is generated
%%\input sn-article.bbl

\end{document}